\documentclass[11pt, reqno]{amsart}
\usepackage{amsmath, amsthm, amscd, amsfonts, amssymb, graphicx, color}
\usepackage[bookmarksnumbered,plainpages]{hyperref}
\usepackage[all]{xy}

\newtheorem{theorem}{Theorem}[section]
\newtheorem{lemma}[theorem]{Lemma}
\newtheorem{proposition}[theorem]{Proposition}
\newtheorem{corollary}[theorem]{Corollary}
\usepackage{enumerate}
\theoremstyle{definition}
\newtheorem{definition}[theorem]{Definition}
\newtheorem{example}[theorem]{Example}

\newtheorem{Theorem}{\quad Theorem}[section]

\newtheorem{remark}[Theorem]{\quad Remark}
\numberwithin{equation}{section}
\input{mathrsfs.sty}
\begin{document}
\setcounter{page}{1}

\title[Numerical radius parallelism]
{Numerical radius parallelism of Hilbert space operators}

\author[M. Mehrazin, M. Amyari, A. Zamani]
{Marzieh Mehrazin$^1$, Maryam Amyari$^{2*}$ \MakeLowercase{and} Ali Zamani$^3$}

\address{$^1,^2$ Department of Mathematics, Mashhad Branch,
Islamic Azad University, Mashhad, Iran}
\email{marzie\_mehrazin@yahoo.com}
\email{amyari@mshdiau.ac.ir and maryam\_amyari@yahoo.com}

\address{$^3$ Department of Mathematics, Farhangian University, Iran}
\email{zamani.ali85@yahoo.com}

\subjclass[2010]{Primary 46B20; Secondary 47L05, 47A12.}

\keywords{operator norm, parallelism, numerical radius.\\
$*$Corresponding author}

\begin{abstract}
In this paper, we introduce a new type of parallelism
for bounded linear operators on a Hilbert space
$\big(\mathscr{H}, \langle \cdot ,\cdot \rangle\big)$
based on numerical radius.
More precisely, we consider
operators $T$ and $S$ which satisfy $\omega(T + \lambda S) = \omega(T)+\omega(S)$
for some complex unit $\lambda$.
We show that $T \parallel_{\omega} S$ if and only if
there exists a sequence of unit vectors $\{x_n\}$ in $\mathscr{H}$ such that
\begin{align*}
\lim_{n\rightarrow\infty} \big|\langle Tx_n, x_n\rangle\langle Sx_n, x_n\rangle\big| = \omega(T)\omega(S).
\end{align*}
We then apply it to give some applications.
\end{abstract} \maketitle
\section{Introduction and preliminaries}
Let $\mathbb{B}(\mathscr{H})$ denote the $C^{\ast}$-algebra of all bounded
linear operators on a complex Hilbert space $\mathscr{H}$ with an inner
product $\langle \cdot ,\cdot \rangle$ and the corresponding norm
$\|\cdot\| $. The symbol $I$ stands for the identity operator
on $\mathscr{H}$.
We shall identify $\mathbb{B}(\mathbb{C}^n)$ and the algebra of all
complex $n\times n$ matrices in the usual way.

Recall that the numerical range of $T\in\mathbb{B}(\mathscr{H})$ is the set
\begin{equation*}
W(T) = \big\{\langle Tx, x\rangle:x\in\mathscr{H},\|x\| = 1\big\},
\end{equation*}
and the numerical radius of $T$ is
\begin{equation*}
\omega(T) = \sup \big\{|\mu|: \mu\in W(T)\big\}.
\end{equation*}
The concepts of numerical range and numerical radius play important roles in many different
areas, especially mathematics and physics (see \cite{G.R, H.J}).

It is known that $W(T)$ is a nonempty bounded convex subset of $\mathbb{C}$,
and $\omega(T)$ is a norm on $\mathbb{B}(\mathscr{H})$ satisfying
$\frac{1}{2}\|T\| \leq \omega(T)\leq \|T\|$,
where $\|T\|$ denotes the operator norm of $T$.
If $T$ is self-adjoint, then $\omega(T) = \|T\|$ and if $T^2 = 0$,
then $\omega(T) = \frac{1}{2}\|T\|$ (see e.g., \cite{Dr} and \cite{G.R}).

Let us recall that by \cite[Lemma 3.2]{H.H.Z} we have
$\omega(x\otimes y) = \frac{1}{2}\left(|\langle x, y\rangle| + \|x\otimes y\|\right)$
and $\|x\otimes y\| = \|x\|\|y\|$, for all $x, y\in\mathscr{H}$.
Here, $x\otimes y$ denotes the rank one operator in $\mathbb{B}(\mathscr{H})$
defined by $(x\otimes y)(z) : = \langle z, y\rangle x$ for all $z\in\mathscr{H}$.
Also, we shall use $[x]$ to denote the linear space spanned by
vector $x\in \mathscr{H}$. 
For more material about the numerical radius and
other results on numerical radius inequality,
see, e.g., \cite{Dr, G.R, H.H.Z, H.J, K.M.Y, Y}, and the references therein.

Now, let $(\mathscr{X}, \|\cdot\|)$ be a normed space.
An element $x\in \mathscr{X}$ is said to be norm--parallel to another element
$y\in \mathscr{X}$ (see \cite{S, Z.M.1}), in short
$x\parallel y$, if
\begin{align*}
\|x+\lambda y\|=\|x\|+\|y\| \qquad \mbox{for some}\,\,
\lambda\in\mathbb{T}.
\end{align*}
Here, as usual, $\mathbb{T}$ is the unit cycle of the complex plane, i.e.
$\mathbb{T}=\{\mu\in\mathbb{C}: \,|\mu|=1\}$.
In the framework of inner product spaces, the norm--parallel relation
is exactly the usual vectorial parallel relation, that is,
$x\parallel y$ if and only if $x$ and $y$ are linearly dependent.
In the setting of normed linear spaces, two linearly
dependent vectors are norm--parallel, but the converse is false in general.
To see this consider the vectors $(1, 0)$ and $(1, 1)$ in the space
$\mathbb{C}^2$ with the max--norm.
Notice that the norm--parallelism is symmetric and $\mathbb{R}$-homogenous,
but is not transitive that is $x \|y$ and $y \| z$ does not imply $x \| z$ in general; 
see \cite[Example 2.7]{Z.M.2}, unless $\mathscr{X}$ is smooth at $y$; see \cite[Theorem 3.1]{W}).

Some characterizations of the norm--parallelism for operators on various Banach spaces
and elements of an arbitrary Hilbert $C^*$-module were given in \cite{B.C.M.W.Z, G, M.S.P, W, Z, Z.M.1, Z.M.2}.

In particular, for $T, S\in\mathbb{B}(\mathscr{H})$, it was proved  \cite[Theorem 3.3]{Z.M.1}
that $T\parallel S$ if and only if there exists a sequence of unit
vectors $\{x_n\}$ in $\mathscr{H}$ such that
\begin{align}\label{id.001}
\lim_{n\rightarrow \infty } \big|\langle Tx_n, Sx_n\rangle\big| = \|T\|\,\|S\|.
\end{align}

Now, let us introduce a new type of parallelism for Hilbert space operators based on numerical radius.
\begin{definition}\label{de.01}
An element $T\in\mathbb{B}(\mathscr{H})$ is called the numerical radius parallel
to another element $S \in\mathbb{B}(\mathscr{H})$, denoted by $T \parallel_{\omega} S$, if
\begin{align*}
\omega(T + \lambda S) = \omega(T)+\omega(S) \qquad \mbox{for some}\,\, \lambda\in\mathbb{T}.
\end{align*}
\end{definition}
If $T,S\in \mathbb{B}(\mathscr{H})$ are linearly dependent,
then there exists $\alpha\in\mathbb{C}\setminus \{0\}$ such that $S=\alpha T$.
By letting $\lambda=\frac{\overline{\alpha}}{|\alpha|},$ we have
\begin{align*}
\omega(T + \lambda S) = \omega(T + |\alpha| T) = (1 + |\alpha|)\omega(T) = \omega(T)+\omega(S),
\end{align*}
and hence $T \parallel_{\omega} S$.
Therefore we observe that two linearly dependent operators are numerical radius parallel,
but the converse is false in general.
For instance, consider the matrices
$T = \begin{bmatrix}
0 & 1\\
0 & 0
\end{bmatrix}$
and
$I = \begin{bmatrix}
1 & 0\\
0 & 1
\end{bmatrix}$.
Then $T$ and $I$ are linearly independent. But
for $\lambda=1$, a straightforward computation shows that
\begin{align*}
\omega(T + \lambda I) = \omega(T)+\omega(I) = \frac{5}{2},
\end{align*}
and so $T \parallel_{\omega} I$.
Notice that the numerical radius parallelism is a reflexive and symmetric relation.
The following example shows that it is not transitive.
\begin{example}
Let
$S =\begin{bmatrix}
1 & 0\\
0 & -1
\end{bmatrix}$,
$I =\begin{bmatrix}
1 & 0\\
0 & 1
\end{bmatrix}$
and
$R =\begin{bmatrix}
0 & 1\\
0 & 0
\end{bmatrix}$.
Then for $\lambda=1$, simple computations show that
\begin{align*}
\omega(S + \lambda I) = \omega(S)+\omega(I) = 2
\end{align*}
and
\begin{align*}
\omega(I + \lambda R) = \omega(I)+\omega(R) = \frac{3}{2}.
\end{align*}
Hence $S \parallel_{\omega} I$ and $I \parallel_{\omega} R$.
But for every $\lambda\in\mathbb{T}$, we have
\begin{align*}
\omega(S + \lambda R) = \frac{\sqrt{4 + |\lambda|^2}}{2} \neq \frac{3}{2} = \omega(S)+\omega(R).
\end{align*}
Thus $S \not\parallel_{\omega} R$.
\end{example}

In the next section, we study the numerical radius parallelism $\parallel_{\omega}$.
We establish some general properties of this relation.
We show that $T \parallel_{\omega} S$ if and only if
there exists a sequence of unit vectors $\{x_n\}$ in $\mathscr{H}$ such that
\begin{align*}
\lim_{n\rightarrow\infty} \big|\langle Tx_n, x_n\rangle\langle Sx_n, x_n\rangle\big| = \omega(T)\omega(S).
\end{align*}
We then apply it to prove that $T\parallel_{\omega} I$ for all $T\in \mathbb{B}(\mathscr{H})$.
Moreover, we prove that if $\dim \mathscr{H}\geq 3$
and $\langle Tx, y\rangle = 0$ for all $x\in \mathscr{H}$ and $y\in {[x]}^{\perp}$,
then $T$ is a scaler multiple of the identity $I$ if and only if 
$T\parallel_{\omega} S$ for all rank one operator $S$.
\section{Main results}
We begin with some properties of the numerical radius parallelism for Hilbert space operators.
\begin{proposition}\label{pr.01}
Let $T,S\in \mathbb{B}(\mathscr{H})$. Then the following conditions are equivalent:
\begin{itemize}
\item[(i)] $T \parallel_{\omega} S$.
\item[(ii)] $T^* \parallel_{\omega} S^*$.
\item[(iii)] $\gamma T \parallel_{\omega} \gamma S
\qquad (\gamma\in\mathbb{C}\smallsetminus\{0\})$.
\item[(iv)] $\alpha T \parallel_{\omega} \beta S
\qquad (\alpha, \beta\in\mathbb{R}\smallsetminus\{0\})$.
\end{itemize}
\end{proposition}
\begin{proof}
The equivalences (i)$\Leftrightarrow$(ii)$\Leftrightarrow$(iii)
immediately follow from the definition of numerical radius parallelism.

(i)$\Leftrightarrow$(iv) Let $T \parallel_{\omega} S$.
Hence $\omega(T + \lambda S) = \omega(T) + \omega(T)$ for some $\lambda\in\mathbb{T}$.
Suppose that $\alpha, \beta\in\mathbb{R}\smallsetminus\{0\}$.
We can assume that $\alpha \geq \beta>0$. We therefore have
\begin{align*}
\omega(\alpha T) + \omega(\beta S)&\geq \omega\big(\alpha T + \lambda(\beta S)\big)
\\&= \omega\big(\alpha(T + \lambda S) - (\alpha-\beta)(\lambda S)\big)
\\&\geq \omega\big(\alpha(T + \lambda S)\big) - \omega\big((\alpha-\beta)\lambda S\big)
\\&= \alpha\omega(T + \lambda S) - (\alpha-\beta)\omega(S)
\\&=\alpha\big(\omega(T) + \omega(S)\big) - (\alpha-\beta)\omega(S)
\\&= \omega(\alpha T) + \omega(\beta S),
\end{align*}
whence $\omega\big(\alpha T + \lambda(\beta S)\big) = \omega(\alpha T) + \omega(\beta S)$.
Thus $\alpha T \parallel_{\omega} \beta S$.

The converse is obvious.
\end{proof}
In the following result we characterize the numerical radius parallelism
for Hilbert space operators.
\begin{theorem}\label{th.1}
Let $T,S\in \mathbb{B}(\mathscr{H})$. Then the following conditions are equivalent:
\begin{itemize}
\item[(i)] $T \parallel_{\omega} S$.
\item[(ii)] There exists a sequence of unit vectors $\{x_n\}$ in $\mathscr{H}$ such that
\begin{align*}
\lim_{n\rightarrow\infty} \big|\langle Tx_n, x_n\rangle\langle Sx_n, x_n\rangle\big| = \omega(T)\omega(S).
\end{align*}
\end{itemize}
In addition, if $\{x_n\}$ is a sequence of unit vectors in $\mathscr{H}$ satisfying $(ii)$,
then it also satisfies $\displaystyle{\lim_{n\rightarrow\infty}}|\langle Tx_n,x_n\rangle| = \omega(T)$
and $\displaystyle{\lim_{n\rightarrow\infty}}|\langle Sx_n,x_n\rangle| = \omega(S)$.
\end{theorem}
\begin{proof}
(i)$\Longrightarrow$(ii)
Let $T \parallel_{\omega} S$. Then there exists $\lambda\in\mathbb{T}$ such that
$\omega(T+\lambda S) = \omega(T)+\omega(S)$.
Since
\begin{align*}
\omega(T+\lambda S)=\sup\Big\{\big|\langle(T+\lambda S)x, x\rangle\big|: x\in \mathscr{H},\|x\|=1\Big\},
\end{align*}
there exists a sequence of unit vectors $\{x_n\}$ such that
\begin{align*}
\lim_{n\rightarrow\infty} \big|\langle(T+\lambda S)x_n, x_n\rangle\big| = \omega(T+\lambda S).
\end{align*}
Since $|\langle Tx_n,x_n\rangle|\leq \omega(T)$ and $|\langle Sx_n,x_n\rangle|\leq \omega(S)$
for every $n$, we have
\begin{align*}
\big(\omega(T)+\omega(S)\big)^2 & = \omega^2(T+\lambda S)
\\& = \lim_{n\rightarrow\infty} \big|\langle(T+\lambda S)x_n, x_n\rangle\big|^2
\\& = \lim_{n\rightarrow\infty}\Big(|\langle Tx_n,x_n\rangle|^2
+ 2\mbox{Re}\big(\lambda\langle Tx_n,x_n\rangle \langle Sx_n,x_n\rangle\big)
+ |\langle S x_n,x_n\rangle|^2\Big)
\\&\leq\lim_{n\rightarrow\infty}\Big(|\langle Tx_n,x_n\rangle|^2
+2\big|\langle Tx_n,x_n\rangle\langle Sx_n,x_n\rangle\big|+|\langle Sx_n,x_n|^2\Big)
\\&\leq \omega^2(T)+2\lim_{n\rightarrow\infty}\big|\langle Tx_n,x_n\rangle\langle Sx_n,x_n\rangle\big|+\omega^2(S)
\\&\leq\omega^2(T)+2\omega(T)\omega(S)+\omega^2(S)= \big(\omega(T)+\omega(S)\big)^2.
\end{align*}
From this it follows that
\begin{align*}
\lim_{n\rightarrow\infty}\big|\langle Tx_n,x_n\rangle\langle Sx_n,x_n\rangle\big| = \omega(T)\omega(S).
\end{align*}

(ii)$\Longrightarrow$(i) Suppose (ii) holds. Then
there exist a sequence of unit vectors $\{x_n\}$ in $\mathscr{H}$
and $\lambda\in\mathbb{T}$ such that
\begin{align*}
\lim_{n\rightarrow\infty} \lambda\langle Tx_n,x_n\rangle \langle Sx_n,x_n\rangle = \omega(T)\omega(S).
\end{align*}
It follows from
\begin{align*}
\omega(T)\omega(S)=\lim_{n\rightarrow\infty}\big|\langle Tx_n,x_n\rangle\langle Sx_n,x_n\rangle\big|
\leq \lim_{n\rightarrow\infty}|\langle Tx_n,x_n\rangle|\omega(S)\leq \omega(T)\omega(S)
\end{align*}
that $\displaystyle{\lim_{n\rightarrow\infty}}|\langle Tx_n,x_n\rangle| = \omega(T)$ and by using a similar argument,
$\displaystyle{\lim_{n\rightarrow\infty}}|\langle Sx_n,x_n\rangle| = \omega(S)$.
Since
\begin{align*}
\big|\langle(T+ \lambda S)x_n,x_n\rangle\big|^2
=|\langle Tx_n,x_n\rangle|^2
+ 2\mbox{Re}\big(\lambda\langle Tx_n,x_n\rangle \langle Sx_n,x_n\rangle\big)
+ |\langle Sx_n,x_n\rangle|^2
\end{align*}
for every $n$, from the above equality we get $\omega^2(T+\lambda S) =\omega^2(T)+2\omega(T)\omega(S) + \omega^2(S)$.
Thus $\omega(T+\lambda S) =\omega(T)+\omega(S)$, or equivalently $T \parallel_{\omega} S$.
\end{proof}
Note that, if Hilbert space $\mathscr{H}$ is finite dimensional, then the closed unit ball
$\mathbf{N}_{\mathscr{H}}$ of $\mathscr{H}$ is compact.
So, the continuity of $T: \mathbf{N}_{\mathscr{H}}\to W(T)$, $x\mapsto \langle Tx, x \rangle$
ensures that $W(T)$ is compact and hence $\omega(T)$ takes its supremum. 
Therefore, as an immediate consequence of Theorem \ref{th.1} we have the following result. 
\begin{corollary}
Let $\mathscr{H}$ be a finite dimensional Hilbert space and $T,S\in \mathbb{B}(\mathscr{H})$.
Then the following conditions are equivalent:
\begin{itemize}
\item[(i)] $T \parallel_{\omega} S$.
\item[(ii)] There exists a unit vector $x\in\mathscr{H}$ such that
\begin{align*}
|\langle Tx, x\rangle\langle Sx, x\rangle| = \omega(T)\omega(S).
\end{align*}
\end{itemize}
\end{corollary}
As another consequence of Theorem \ref{th.1} we have the following result.
\begin{corollary}\label{cr.01}
Let $T\in \mathbb{B}(\mathscr{H})$. Then the following statements hold.
\begin{itemize}
\item[(i)] $T \parallel_{\omega}\alpha I~~~~~$ for all $\alpha \in \mathbb{C}$.
\item[(ii)] $T \parallel_{\omega} T^*$.
\end{itemize}
\end{corollary}
\begin{proof}
By definition of $\omega(T)$, there exists a sequence of unit vectors in $\mathscr{H}$
such that
\begin{align*}
\lim_{n\rightarrow\infty}|\langle Tx_n,x_n\rangle| = \omega(T).
\end{align*}
Since $\omega(\alpha I)=|\alpha|$ and $\omega(T^*) = \omega(T)$, we have
\begin{align}\label{cr.011}
\lim_{n\rightarrow\infty}\big|\langle Tx_n,x_n\rangle\langle  (\alpha I)x_n,x_n\rangle\big|
= \lim_{n\rightarrow\infty}|\langle Tx_n,x_n\rangle||\alpha| = \omega(T)|\alpha| = \omega(T) \omega(\alpha I)
\end{align}
and
\begin{align}\label{cr.012}
\lim_{n\rightarrow\infty}\big|\langle Tx_n,x_n\rangle\langle T^*x_n,x_n\rangle\big|
= \lim_{n\rightarrow\infty}|\langle Tx_n,x_n\rangle|^2 = \omega^2(T) = \omega(T)\omega(T^*).
\end{align}
Hence by (\ref{cr.011}), (\ref{cr.012}) and Theorem \ref{th.1},
we obtain $T \parallel_{\omega} I$ and $T \parallel_{\omega} T^*$.
\end{proof}
As a consequence of Theorem \ref{th.1},
we have the following characterization of the numerical radius parallelism for
special type of rank one operators.
\begin{corollary}
Let $x,y\in \mathscr{H}$. Then the following conditions are equivalent:
\begin{itemize}
\item[(i)] $x\otimes x \parallel_{\omega} y\otimes y$.
\item[(ii)] $x\parallel y$.
\end{itemize}
\end{corollary}
\begin{proof}
Suppose that $x\otimes x\|_\omega y\otimes y$. So, by Theorem \ref{th.1},
there exists a sequence of unit vectors $x_n$ in $\mathscr{H}$ such that
\begin{align}\label{cr.021}
\lim_{n\rightarrow\infty}\Big|\langle (x\otimes x)x_n,x_n\rangle\langle(y\otimes y)x_n,x_n\rangle\Big|
=\omega(x\otimes x)\omega(y\otimes y).
\end{align}
On the other side, we have
\begin{align}\label{cr.022}
\langle (x\otimes x)x_n,x_n\rangle\langle(y\otimes y)x_n,x_n\rangle
=\big|\langle x, x_n\rangle\langle x_n,y\rangle\big|^2
= \big|\langle (x\otimes y)x_n, x_n\rangle\big|^2
\end{align}
and
\begin{align}\label{cr.023}
\omega(x\otimes x)\omega(y\otimes y) =\|x\|^2\|y\|^2 = \|x\otimes y\|^2.
\end{align}
By (\ref{cr.021}), (\ref{cr.022}) and (\ref{cr.023}) it follows that
\begin{align*}
\lim_{n\rightarrow\infty} \big|\langle (x\otimes y)x_n, x_n\rangle\big| = \|x\otimes y\|,
\end{align*}
and by (\ref{id.001}) from the above equality we get $x\otimes y\parallel I$.
Now, by \cite[Corollary 2.23]{Z.M.2}, we conclude that $x\parallel y$.

The converse is obvious.
\end{proof}
Let us quote a result from Halmos \cite{Hal}.
\begin{lemma}\cite{Hal}\label{le.011}
Let $T\in \mathbb{B}(\mathscr{H})$. Then $W(T) = \{\mu\}$ if and only if $T = \mu I$.
\end{lemma}
We are now in a position to establish one of our main results.
\begin{theorem}
Let $\mathscr{H}$ be a complex Hilbert space with $\dim \mathscr{H}\geq 3$ 
and $T\in \mathbb{B}(\mathscr{H})$ such that $\langle Tx, y\rangle = 0$ for all $x\in \mathscr{H}$ and $y\in {[x]}^{\perp}$.
Then the following conditions are equivalent:
\begin{itemize}
\item[(i)] $T\parallel_{\omega} S$ for all rank one operator $S\in \mathbb{B}(\mathscr{H})$.
\item[(ii)] $T$ is a scaler multiple of the identity $I$.
\end{itemize}
\end{theorem}
\begin{proof}
(i)$\Longrightarrow$(ii) Suppose (i) holds. Let $x\in \mathscr{H}$ and $\|x\| = 1$.
We can choose $y\in {[x]}^{\perp}$ with $\|y\| = 1$ such that
$\langle T^*x, y\rangle = 0$. Put $S = x\otimes y$ and hence we get $\omega(S) = \frac{1}{2}$.
Now, our assumption implies that $T\parallel_{\omega} x\otimes y$.
So, by Theorem \ref{th.1}, there exists a sequence of unit vectors $\{x_n\}$ in $\mathscr{H}$ such that
\begin{align*}
\lim_{n\rightarrow\infty} \big|\langle Tx_n, x_n\rangle\langle (x\otimes y)x_n, x_n\rangle\big| = \frac{1}{2}\omega(T),
\end{align*}
\begin{align}\label{th.022}
\lim_{n\rightarrow\infty}|\langle Tx_n,x_n\rangle| = \omega(T)
\qquad \mbox{and} \qquad \lim_{n\rightarrow\infty}\big|\langle (x\otimes y)x_n,x_n\rangle\big| = \frac{1}{2}.
\end{align}
Let $x_n = \alpha_n x + \beta_n y + z_n$ such that $\langle z_n, x\rangle = \langle z_n, y\rangle = 0$.
Since sequences $\{\alpha_n\}$ and $\{\beta_n\}$ are both bounded 1
and we can suppose that (after passage to a subsequence) $\displaystyle{\lim_{n\rightarrow\infty}} \alpha_n = \alpha$
and $\displaystyle{\lim_{n\rightarrow\infty}} \beta_n = \beta$ for some $\alpha,~~~\beta$. From (\ref{th.022}) it follows that
\begin{align*}
\frac{1}{2} &= \lim_{n\rightarrow\infty}\big|\langle (x\otimes y)x_n,x_n\rangle\big|
\\& = \lim_{n\rightarrow\infty}\Big|\langle (x\otimes y)(\alpha_n x + \beta_n y + z_n), \alpha_n x + \beta_n y + z_n\rangle\Big|
\\& = \lim_{n\rightarrow\infty}\Big|\langle \langle\alpha_n x + \beta_n y + z_n, y\rangle x, \alpha_n x + \beta_n y + z_n\rangle\Big|
\\& = \lim_{n\rightarrow\infty} |\alpha_n||\beta_n| = |\alpha||\beta|,
\end{align*}
which gives
\begin{align}\label{th.023}
|\alpha||\beta| = \frac{1}{2}.
\end{align}
Further, we have
\begin{align}\label{th.024}
|\alpha_n|^2 + |\beta_n|^2 + \|z_n\|^2
= \langle \alpha_n x + \beta_n y + z_n, \alpha_n x + \beta_n y + z_n\rangle
= \|x_n\|^2 = 1,
\end{align}
and hence $|\alpha|^2 + |\beta|^2 \leq 1$. So, by (\ref{th.023}),
we reach $|\alpha| = |\beta| = \frac{\sqrt{2}}{2}$.
Therefore, form (\ref{th.024}) it follows that $\displaystyle{\lim_{n\rightarrow\infty}}\|z_n\|^2 = 0$,
or equivalently, $\displaystyle{\lim_{n\rightarrow\infty}}z_n = 0$.
Thus
\begin{align}\label{th.025}
\lim_{n\rightarrow\infty}x_n = \lim_{n\rightarrow\infty}(\alpha_n x + \beta_n y + z_n) = \alpha x + \beta y.
\end{align}
By (\ref{th.022}) and (\ref{th.025}) we get
\begin{align*}
\big|\langle T(\alpha x + \beta y), \alpha x + \beta y\rangle\big| = \omega(T).
\end{align*}
Since $\langle T^*x, y\rangle = \langle Tx, y\rangle = 0$, from the above equality we get
\begin{align*}
\big|\frac{1}{2}\langle Tx, x\rangle + \frac{1}{2}\langle Ty, y\rangle\big| = \omega(T),
\end{align*}
and hence
\begin{align*}
\omega(T) = \big|\frac{1}{2}\langle Tx, x\rangle + \frac{1}{2}\langle Ty, y\rangle\big|
\leq |\frac{1}{2}\langle Tx, x\rangle| + \frac{1}{2}\omega(T)
\leq \frac{1}{2}\omega(T) + \frac{1}{2}\omega(T) = \omega(T).
\end{align*}
Thus $|\langle Tx, x\rangle| = \omega(T)$ for all $x\in \mathscr{H}$ with $\|x\| = 1$.
This yields elements in $W(T)$ are of constant modulus. It follows then from the
convexity of $W(T)$ that the set is a singleton.
Hence by Lemma \ref{le.011} $T$ is a scaler multiple of the identity $I$.

(ii)$\Longrightarrow$(i) This implication follows immediately from
Corollary \ref{cr.01}(i).
\end{proof}
\begin{remark}
Let $T, S$ be operators in $\mathbb{B}(\mathscr{H})$ such that $T\parallel S$.
Therefore, $\|T + e^{2i\theta} S\| = \|T\| + \|S\|$ for some $\theta\in\mathbb{R}$.
By \cite[Theorem 2.3]{K.M.Y} it follows that
\begin{align*}
\|T\| + \|S\| = \|T + e^{2i\theta} S\| \leq 2 \omega\Big(\begin{bmatrix}
0 & T\\
e^{-2i\theta}S^* & 0
\end{bmatrix}\Big) \leq \|T\| + \|S\|,
\end{align*}
or equivalently
\begin{align*}
\|T\| + \|S\| \leq 2 \omega\left(\begin{bmatrix}
0 & e^{i\theta}T\\
e^{-i\theta}S^* & 0
\end{bmatrix}\right) \leq \|T\| + \|S\|.
\end{align*}
Thus $\omega\left(\begin{bmatrix}
0 & e^{i\theta}T\\
e^{-i\theta}S^* & 0
\end{bmatrix}\right) = \frac{1}{2}(\|T\| + \|S\|)$ for some $\theta\in\mathbb{R}$.
\end{remark}
\begin{remark}
Let $T, S$ be normal operators in $\mathbb{B}(\mathscr{H})$ such that $T\parallel_{\omega} S$.
So, $\omega(T + \lambda S) = \omega(T) + \omega(S)$ for some $\lambda\in\mathbb{T}$.
Since $\omega(T)=\|T\|$ and $\omega(S)=\|S\|$, we have
\begin{align*}
\omega(T) + \omega(S) = \omega(T + \lambda S)\leq \|T + \lambda S\|\leq \|T\| + \|S\| = \omega(T) + \omega(S).
\end{align*}
Hence $\|T + \lambda S\| = \|T\| + \|S\|$,
that is $T\parallel S.$ Therefore, by (\ref{id.001}), there exists a sequence of unit vectors $\lbrace x_n\rbrace$ in $\mathscr{H}$ such that
\begin{align*}
\lim_{n\rightarrow\infty}|\langle Tx_n,Sx_n\rangle| = \omega(T)\omega(S).
\end{align*}
\end{remark}


\bibliographystyle{amsplain}

\end{document}